\documentclass[12pt]{article}
\usepackage{amssymb, amsmath, amsthm, amsfonts,xcolor, enumerate}
\usepackage{tikz}
\usetikzlibrary{patterns}
\usetikzlibrary{shapes}
\usepackage{fullpage}
\usepackage{hyperref}
\newtheorem{theorem}{Theorem}[section]
\newtheorem{cl}[theorem]{Claim}
\newtheorem{lemma}[theorem]{Lemma}
\newtheorem{prop}[theorem]{Proposition}
\theoremstyle{definition}

\newtheorem{question}[theorem]{Question}
\theoremstyle{definition}

\theoremstyle{definition}

\theoremstyle{definition}

\theoremstyle{definition}

\theoremstyle{definition}

\theoremstyle{definition}

\theoremstyle{definition}

\newcommand{\de}{\delta}

\newcommand{\ga}{\gamma}
\newcommand{\Ga}{\Gamma}
\newcommand{\De}{\Delta}

\newcommand{\cF}{\mathcal{F}}

\def\eps{\varepsilon}

\title{The number of maximal sum-free subsets of integers}
\author{J\'ozsef Balogh,\footnote{Department of Mathematics, University of Illinois, Urbana, IL 61801, USA and Bolyai Institute, University of Szeged, Szeged, Hungary, {\tt jobal@math.uiuc.edu}.
   Research is partially supported by Simons Fellowship, NSF CAREER Grant DMS-0745185, Arnold O. Beckman Research Award (UIUC Campus Research Board 13039) and Marie Curie FP7-PEOPLE-2012-IIF 327763.}
\ Hong Liu,\footnote{Department of Mathematics, University of Illinois, Urbana, IL 61801, USA, {\tt hliu36@illinois.edu}.}
\ Maryam Sharifzadeh\footnote{Department of Mathematics, University of Illinois, Urbana, IL 61801, USA, {\tt sharifz2@illinois.edu}.} 
\  and Andrew Treglown\footnote{School of Mathematics, University of Birmingham, United Kingdom, {\tt a.c.treglown@bham.ac.uk}.}}

\begin{document}
\maketitle

\begin{abstract}
Cameron and Erd\H{o}s~\cite{CE}  raised the question of how many \emph{maximal} sum-free sets there are in $\{1, \dots , n\}$, giving a lower bound of $2^{\lfloor n/4 \rfloor }$. In this paper we prove that there are in fact at most $2^{(1/4+o(1))n}$ maximal sum-free sets in $\{1, \dots , n\}$.
 Our proof makes use of  container and removal lemmas of Green~\cite{G-CE, G-R} as well as a  result of Deshouillers, Freiman, S\'os and Temkin~\cite{DFST} on the structure of sum-free sets.
\end{abstract}

\section{Introduction}
A fundamental notion in combinatorial number theory is that of a sum-free set: A set $S$ of integers is \emph{sum-free} if $x+y \not \in S$ for every $x,y \in S$ (note $x$ and $y$ are not necessarily distinct here). 
The topic of sum-free sets of integers has a long history. Indeed, in 1916 Schur~\cite{schur} proved that, if $n$ is sufficiently large, then any $r$-colouring of  $[n]:=\{1, \dots , n\}$ yields a monochromatic triple $x,y,z$ such that $x+y=z$. 

Note that both the set of odd numbers in $[n]$ and the set $\{\lfloor n/2 \rfloor +1, \dots , n\}$ are maximal sum-free sets. (A sum-free subset of $[n]$ is \emph{maximal} if it is not properly contained in another sum-free subset of $[n]$.) By considering all possible subsets of one of these maximal sum-free sets, we see that $[n]$ contains at least $2^{\lceil n/2 \rceil}$ sum-free sets. Cameron and Erd\H{o}s~\cite{cam1} conjectured that in fact $[n]$ contains only $O (2^{n/2})$ sum-free sets. 
The conjecture was proven independently by Green~\cite{G-CE} and Sapozhenko~\cite{sap}. Recently, a refinement of the Cameron--Erd\H{o}s conjecture was proven in~\cite{abms}, giving an upper bound on the number of sum-free sets in $[n]$ of size $m$ (for each $1 \leq m \leq \lceil n/2 \rceil$).

Let $f(n)$ denote the number of sum-free subsets of $[n]$ and $f_{\max}(n)$ denote the number of maximal sum-free subsets of $[n]$.
Recall that the sum-free subsets of $[n]$ described above lie in just two maximal sum-free sets. This led Cameron and Erd\H{o}s~\cite{CE} to ask whether the number of maximal sum-free subsets of $[n]$ is ``substantially smaller''
than the total number of sum-free sets. In particular, they asked whether $f_{\max} (n) = o(f(n))$ or even $f_{\max} (n) \leq f(n)/ 2^{\eps n}$ for some constant $\eps >0$. \L{u}czak and Schoen~\cite{ls} answered this question, showing that $f_{\max} (n) \leq 2^{n/2-2^{-28}n}$ for sufficiently large $n$. More recently, Wolfovitz~\cite{wolf} proved that $f_{\max} (n) \leq 2^{3n/8+o(n)}$. 

In the other direction, Cameron and Erd\H{o}s~\cite{CE} observed that $f_{\max} (n) \geq 2^{\lfloor n/4 \rfloor }$. Indeed, let $m=n$ or $m=n-1$, whichever is even. Let $S$ consist of $m$ together with precisely
one number from each pair  $\{x,m-x\}$ for odd $x <m/2$. Then $S$ is sum-free. Moreover, although $S$ may not be maximal, no further odd numbers less than $m$ can be added, so distinct $S$ lie in distinct maximal sum-free subsets of $[n]$.

We prove that this lower bound is in fact, `asymptotically', the correct bound on $f_{\max} (n)$.

\begin{theorem}\label{thm-main}
There are at most $2^{(1/4+o(1))n}$ maximal sum-free sets in $[n]$. That is,
$$f_{\max}(n)=2^{(1/4+o(1))n}.$$
\end{theorem} 
The proof of Theorem~\ref{thm-main} makes use of `container' and `removal' lemmas of Green~\cite{G-CE, G-R} as well as a  result of Deshouillers, Freiman, S\'os and Temkin~\cite{DFST} on the structure of sum-free sets (see Section~\ref{sec-pre} for an overview of the proof).

Next we provide another collection of maximal sum-free sets in $[n]$. Suppose that $4|n$ and set $I_1:=\{n/2+1, \ldots, 3n/4\}$ and $I_2:=\{3n/4+1, \ldots,n\}$. First choose the element $n/4$ and a set $S\subseteq I_2$. Then for every $x\in I_2\setminus S$, choose $x-n/4\in I_1$. The resulting set is sum-free but may not be maximal. However, no further element in $I_2$ can be added, thus distinct $S$ lie in distinct maximal sum-free sets in $[n]$. There are $2^{|I_2|}=2^{n/4}$ ways to choose $S$.

It would be of interest to establish whether $f_{\max} (n)=O (2^{n/4})$.
\begin{question}
Does $f_{\max} (n)=O(2^{n/4})$?
\end{question}
A solution to the question is a current work in progress.
In a forthcoming paper~\cite{agroup} we consider the analogous problem for maximal sum-free sets in abelian groups.

\vspace{2mm}
\noindent\textbf{Notation:} Given a set $A\subseteq [n]$, denote by 
$f_{\max}(A)$ the number of maximal sum-free subsets of $[n]$ that lie in $A$ and by $\min(A)$ the minimum element of $A$. Let $1\le p<q\le n$ be integers, denote $[p,q]:=\{p,p+1,\ldots,q\}$. Denote by $E$ the set of all even numbers in $[n]$ and by $O$ the set of all odd numbers in $[n]$. A triple $x,y,z\in [n]$ is called a \emph{Schur triple} if $x+y=z$ (here $x=y$ is allowed).

Throughout, all graphs considered are simple unless stated otherwise. We say that a graph $G$ is a \emph{graph possibly with loops} if $G$ can be obtained from a simple graph by adding at most one loop at each vertex. 
Given a vertex $x$ in $G$, we write $\deg _G (x)$ for the \emph{degree of  $x$ in $G$}. Note that a loop at $x$ contributes two to the degree of $x$. We write $\delta (G)$ for the \emph{minimum degree of $G$} and $\Delta (G)$ for the \emph{maximum degree of $G$}.
Given a graph $G$, denote by ${\rm{MIS}}(G)$ the number of maximal independent sets in $G$. Given $T\subseteq V(G)$, denote by $\Ga(T)$ the external neighbourhood of $T$, i.e.~$\Ga(T):=\{v\in V(G)\setminus T: \exists u\in T, uv\in E(G)\}$. Denote by $G[T]$ the induced subgraph of $G$ on the vertex set $T$ and  let  $G\setminus T$ denote the induced subgraph of $G$ on the vertex set $V(G)\setminus T$. Denote by $E(T)$ the set of edges in $G$ spanned by $T$ and by $E(T,V(G)\setminus T)$ the set of edges in $G$ with exactly one vertex in $T$.

\section{Overview of the proof and preliminary  results}\label{sec-pre}
\subsection{Proof overview}
We prove Theorem~\ref{thm-main} in Section~\ref{sec-main}.
A key tool in the proof  is the following container lemma of Green~\cite{G-CE} for sum-free sets.
The first container-type result in the area (for counting sum-free subsets of $\mathbb Z_p$) was given by Green and Ruzsa~\cite{new}.
\begin{lemma}[Proposition 6 in~\cite{G-CE}]\label{lem-container}
There exists a family $\cF$ of subsets of $[n]$ with the following properties.

(i) Every member of $\cF$ has at most $o(n^2)$ Schur triples.

(ii) If $S\subseteq [n]$ is sum-free, then $S$ is contained in some member of $\cF$.

(iii) $|\cF|=2^{o(n)}$.

(iv) Every member of $\cF$ has size at most $(1/2+o(1))n$.
\end{lemma}
We refer to the elements of $\cF$ from Lemma~\ref{lem-container} as \emph{containers}. 
In~\cite{G-CE}, condition (iv) was not  stated explicitly. However, it follows immediately from (i) by, for example, applying Theorem~\ref{thm-structure} and Lemma~\ref{lem-removal} below.
Lemma~\ref{lem-container} can also be derived from a general theorem of Balogh, Morris and Samotij~\cite{BMS}, and independently Saxton and Thomason~\cite{ST} with better bounds in (i) and (iii).

Note that conditions (ii) and (iii) in Lemma~\ref{lem-container} imply that, to prove Theorem~\ref{thm-main}, it suffices to show that every  member of $\cF$ contains at most $2^{n/4+o(n)}$ maximal sum-free subsets of $[n]$. For this purpose, we need to get a handle on the structure of the containers; this is made precise in Lemma~\ref{lem-structure} below. 
The following theorem of Deshouillers, Freiman, S\'os and Temkin~\cite{DFST} provides a structural characterisation of the sum-free sets in $[n]$.

\begin{theorem}\label{thm-structure}
Every sum-free set $S$ in $[n]$ satisfies at least one of the following conditions:

(i) $|S|\le {2n}/{5}+1$;

(ii) $S$ consists of odd numbers;

(iii) $|S|\le\min(S)$.
\end{theorem}
We also need the following removal lemma of Green~\cite{G-R} for sum-free sets. (A simpler proof of Lemma~\ref{lem-removal} was later given by Kr\'al', Serra and Vena~\cite{ksv}.)
\begin{lemma}[Corollary 1.6 in~\cite{G-R}]\label{lem-removal}
Suppose that $A\subseteq [n]$ is a set containing $o(n^2)$ Schur triples. Then, there exist $B$ and $C$ such that $A=B\cup C$ where $B$ is sum-free and $|C|=o(n)$.
\end{lemma}
Together, Theorem~\ref{thm-structure} and Lemma~\ref{lem-removal} yield the following structural result on containers of size close to $n/2$.

\begin{lemma}\label{lem-structure}
If $A\subseteq [n]$ has $o(n^2)$ Schur triples and $|A|=(\frac{1}{2}-\ga)n$ with $\ga =\ga (n)\leq  1/11$, then one of the following conditions holds.

(a) All but  $o(n)$ elements of $A$ are contained in the interval $[(1/2-\ga)n, n]$.

(b) Almost all elements of $A$ are odd, i.e.~$|A\setminus O|=o(n)$.
\end{lemma}

\begin{proof}
Apply Lemma~\ref{lem-removal} to $A$; we have $A=B\cup C$ with $B$ sum-free and $|C|=o(n)$. Apply Theorem~\ref{thm-structure} to $B$. Alternative (i) is impossible, since $|B|\ge (1-o(1))|A|>2n/5+1$. If alternative (ii) occurs, then we have $|A\setminus O|\leq|C|=o(n)$. If alternative (iii) occurs, then $\min(B)\ge |B|\ge (1/2-\ga-o(1))n$. So all but except $o(n)$ elements of $A$ are contained in $[(1/2-\ga)n,n]$.
\end{proof}
We remark that Lemma~\ref{lem-structure} was already essentially proven in~\cite{G-CE} (without applying Lemma~\ref{lem-removal}). 
Note that $\gamma$ could be negative in Lemma~\ref{lem-structure}. The upper bound $1/11$ on $\gamma$ here can be relaxed  to any constant smaller than $1/10$ (but not to a constant bigger than $1/10$).
Roughly speaking,
Lemma~\ref{lem-structure} implies that every container $A\in \cF$ is such that  (a) most elements of $A$ lie in $[n/2, n]$, (b) most elements of $A$ are odd or (c) $|A|$ is significantly smaller than $n/2$.
Thus, the proof of Theorem~\ref{thm-main} splits into three cases depending on the structure of our container. In each case, we give an upper bound on the number of maximal sum-free sets in a container by
counting the number of maximal independent sets in various auxiliary graphs. (Similar techniques were used in~\cite{wolf}, and in the graph setting in~\cite{sar}.)
In the following subsection we collect together a number of results that are useful for this.

\subsection{Maximal independent sets in graphs}
Moon and Moser~\cite{MM} showed that for any graph $G$, ${\rm{MIS}}(G)\le 3^{|G|/3}$. We will need a looped version of this statement. Since any vertex with a loop cannot be in an independent set, the following statement is an immediate consequence of Moon and Moser's result.

\begin{prop}\label{lem-mm}
Let $G$ be a graph  possibly with loops. Then
$${\rm{MIS}}(G)\le 3^{|G|/3}.$$
\end{prop}

When a graph is triangle-free, the bound in Proposition~\ref{lem-mm} can be improved significantly. A result of Hujter and Tuza~\cite{HT} states that for any triangle-free graph $G$,
 \begin{align}\label{htnew}
 {\rm{MIS}}(G)\le 2^{|G|/2}.
 \end{align} The following lemma is a slight modification of this result for graphs with  `few' triangles.
\begin{lemma}\label{lem-ht}
Let $G$ be a graph possibly with loops. If there exists a set $T$ such that $G\setminus T$ is triangle-free, then
$${\rm{MIS}}(G)\le 2^{|G|/2+|T|/2}.$$
\end{lemma}
\begin{proof}
Every maximal independent set in $G$ can be obtained in the following two steps: 

(1) Choose an independent set $S\subseteq T$.

(2) Extend $S$ in $V(G)\setminus T$, i.e.~choose a set $R\subseteq V(G)\setminus T$ such that $R\cup S$ is a maximal independent set in $G$.

Note that although every maximal independent set in $G$ can be obtained in this way, it is not necessarily the case that given an arbitrary independent set $S \subseteq T$, there exists a set $R \subseteq V(G) \setminus T$ such that $R \cup S$ is  a maximal independent set in $G$.
Notice that if $R\cup S$ is maximal,  $R$ is also a maximal independent set in $G\setminus\{T\cup \Ga(S)\}$. The number of choices for $S$ in (1) is at most $2^{|T|}$. Since $G\setminus\{T\cup \Ga(S)\}$ is triangle-free, by the Hujter--Tuza bound, the number of extensions in (2) is at most $2^{(|G|-|T|)/2}$. Thus, we have ${\rm{MIS}}(G)\le 2^{|T|}\cdot 2^{(|G|-|T|)/2}=2^{|G|/2+|T|/2}$.
\end{proof}

The following lemma gives an improvement on Proposition~\ref{lem-mm} for graphs that are `not too sparse and almost regular'. The proof uses an elegant and simple idea of Sapozhenko~\cite{S}, see~\cite{IK} for a closely-related result.
\begin{lemma}\label{lem-ik}
Let $k\ge 1$ and let $G$ be a graph on $n$ vertices possibly with loops. 
Suppose that $\De(G)\le k\de(G)$ where $\de(G)\geq f(n)$ for some real valued function $f$ with $f(n)\rightarrow \infty$ as $n \rightarrow \infty$.  Then
$${\rm{MIS}}(G)\le 3^{\left(\frac{k}{k+1}\right)\frac{n}{3}+o(n)}.$$
\end{lemma}
\begin{proof}
Fix a maximal independent set $I$ in $G$ and set $b:=\delta(G)^{1/2}$. We will repeat the following process as many times as possible. Let $V_1:=V(G)$. At the $i$-th step, for $i\ge 1$, choose $v_i\in V_i\cap I$ such that $\deg_{G[V_i]}(v_i)\ge b$ and set $V_{i+1}:=V_i\setminus(\{v_i\}\cup \Gamma(v_i))$. This process is repeated $j \leq n/b$ times. Let $U:=V_{j+1}$ be the resulting set. Define $Z:=\{v\in U: \deg_{G[U]}(v)<b\}$. Notice that $\deg_{G[U]}(v)<b$ for all $v\in I\cap U$, hence $I\cap U\subseteq Z$. We have
$$\delta(G)\cdot|Z|\le\sum_{v\in Z}\deg(v)=2|E(Z)|+|E(Z,V\setminus Z)|\le b|Z|+\Delta(G)\cdot(n-|Z|).$$
Hence,
\begin{align}\label{eq}
|Z|\le\frac{\Delta(G)\cdot n}{\delta(G)+\Delta(G)-b}\le\frac{k}{k+1}n+\frac{2n}{b}.
\end{align}

By construction of $U$, no vertex in $I \setminus U$ has a neighbour in $U$. So as $Z \subseteq U$, no vertex in $Z$ is adjacent to $I\setminus U$. Together with the fact that $I$ is maximal, this implies that
$I\cap U$ is a maximal independent set in $G[Z]$. By the above process, every maximal independent set $I$ in $G$ is determined by a set $I\setminus U$ of  at most  $n/b$ vertices and a maximal independent set in $G[Z]$. Note that $n/b=o(n)$. Thus, Proposition~\ref{lem-mm} and (\ref{eq}) imply that
\begin{align}
{\rm{MIS}}(G)\le \sum_{0 \le i\le n/b}{n\choose i}3^{{\left(\frac{k}{k+1}\right)\frac{n}{3}} + \frac{2n}{3b}}{\le} \, 3^{\left(\frac{k}{k+1}\right)\frac{n}{3}+o(n)}.
\end{align}
\end{proof}
Note that one could relax the minimum degree condition in Lemma~\ref{lem-ik} to (for example) a large constant, at the expense of a worse upper bound on ${\rm MIS}(G)$. However, Lemma~\ref{lem-ik} in its current form suffices for our applications.

\section{Proof of Theorem~\ref{thm-main}}\label{sec-main}
Let $\cF$ be the family of containers obtained from Lemma~\ref{lem-container}. 
Recall that given a set $A\subseteq [n]$,
$f_{\max}(A)$ denotes the number of maximal sum-free subsets of $[n]$ that lie in $A$.
Since every sum-free subset of $[n]$ is contained in some member of $\cF$ and $|\cF|=2^{o(n)}$, it suffices to show that $f_{\max}(A) \leq 2^{(1/4+o(1))n}$ for every container $A\in\cF$.

Lemmas~\ref{lem-container} and~\ref{lem-structure} imply that every container $A\in\cF$ satisfies at least one of the following conditions: 

(a) $|A|\leq (1/2-1/11)n \leq 0.45n$ 

or one of the following holds for some $-o(1) \leq \gamma =\gamma (n) \leq 1/11$:

(b) $|A|=\left(\frac{1}{2}-\ga\right)n$ and $|A\cap [(1/2-\ga)n]|=o(n)$; 

(c) $|A|=\left(\frac{1}{2}-\ga\right)n$ and $|A\setminus O|= o(n)$.

\smallskip

\noindent
We deal with each of the three cases separately.

\medskip

For any subsets $B, S\subseteq [n]$, let $L_S[B]$ be the \emph{link graph of $S$ on $B$} defined as follows. The vertex set of $L_S[B]$ is $B$. The edge set of $L_S[B]$ consists of the following two types of edges:

(i) Two vertices $x$ and $y$ are adjacent if there exists an element $z\in S$ such that $\{x,y,z\}$ 

forms a Schur triple; 

(ii) There is a  loop at a vertex $x$ if $\{x,x, z\}$ forms a Schur triple for some $z \in S$ or if 

$\{x, z,z'\}$ 
forms a Schur triple for some $z, z'\in S$.

\smallskip

The following simple result will be applied in all three cases of our proof.

\begin{lemma}\label{claim-mis}
Suppose that $B,S $ are both sum-free subsets of $[n]$. If $I\subseteq B$ is  such that $S\cup I$ is a maximal sum-free subset of $[n]$, then $I$ is a maximal independent set in $G:=L_S[B]$.
\end{lemma}
\begin{proof}
First notice that $I$ is an independent set in $G$, since otherwise $S\cup I$ is not sum-free. Suppose to the contrary that there exists a vertex $v\not\in I$ such that $I':=I\cup\{v\}$ is still an independent set in $G$. Then since $I'\subseteq B$ is sum-free, the definition of $G$ implies that $S\cup I'$ is a sum-free set in $[n]$ containing $S\cup I$, a contradiction to the maximality of $S\cup I$.
\end{proof}

\subsection{Small containers}
The following lemma deals with containers of `small' size.

\begin{lemma}\label{lem-non-ext}
If $A\in\cF$ has size at most $0.45n$, then $f_{\max}(A)=o(2^{n/4})$.
\end{lemma}
\begin{proof}
Lemma~\ref{lem-container} (i) implies that we can apply Lemma~\ref{lem-removal} to $A$ to obtain that $A=B\cup C$ where $B$ is sum-free and $|C|=o(n)$. Notice crucially that every maximal sum-free subset of $[n]$ in $A$ can be built in the following two steps: 

(1) Choose a sum-free set $S$ in $C$; 

(2) Extend $S$ in $B$ to a maximal one.

\noindent
(As in Lemma~\ref{lem-ht}, note that it is not necessarily the case that given an arbitrary sum-free set $S \subseteq C$, there exists a set $R \subseteq B$ such that $R \cup S$ is  a maximal sum-free set in $[n]$.)

The number of choices for $S$ is at most $2^{|C|}=2^{o(n)}$. For a fixed $S$, denote by $N(S,B)$ the number of extensions of $S$ in $B$ in Step (2). It suffices to show that for any given sum-free set $S\subseteq C$, $N(S,B)\leq 2^{0.249n}$. Let $G:=L_S[B]$ be the link graph of $S$ on $B$. 
Since $|A|\le 0.45n$ and $S$ and $B$ are sum-free, Lemma~\ref{claim-mis} and Proposition~\ref{lem-mm} imply that
$$N(S,B)\le {\rm{MIS}}(G)\le 3^{|B|/3}\le 3^{|A|/3}\le 3^{0.45n/3}\ll 2^{0.249n}.$$
\end{proof}


\subsection{Large containers}
We now turn our attention to containers of relatively large size.

\begin{lemma}\label{lem-interval}
Let $-o(1) \leq \ga =\ga (n) \leq 1/11$. If $A\subseteq [n]$ has $o(n^2)$ Schur triples, $|A|=\left(\frac{1}{2}-\ga\right)n$ and $|A\cap [(1/2-\ga)n]|=o(n)$, then
$$f_{\max}(A)\le 2^{(1/4+o(1))n}.$$
\end{lemma}
\begin{proof}
Let $A\in\cF$ be as in the statement of the lemma. Let $A_1:=A\cap [\lfloor n/2 \rfloor]$ and $A_2:=A\setminus A_1$. Since $|A\cap [(1/2-\ga)n]|=o(n)$, we have that $|A_1|\le (\ga+o(1))n$.
Every maximal sum-free subset of $[n]$ in $A$ can be built from choosing a sum-free set $S\subseteq A_1$ and extending $S$ in $A_2$. The number of choices for $S$ is at most $2^{|A_1|}$. 

Let $G:=L_S[A_2]$ be the link graph of $S$ on vertex set $A_2$. Since $S$ and $A_2$ are sum-free, Lemma~\ref{claim-mis} implies that $N(S,A_2)\le {\rm{MIS}}(G)$. Notice that $G$ is triangle-free. Indeed, suppose to the contrary that $z>y>x>n/2$ form a triangle in $G$. Then there exists $a,b,c\in S$ such that $z-y=a, y-x=b$ and $z-x=c$, which implies $a+b=c$ with $a,b,c\in S$. This is a contradiction to $S$ being sum-free. Thus by (\ref{htnew}) we have $N(S,A_2)\le {\rm{MIS}}(G)\le 2^{|A_2|/2}$. Then we have 
$$f_{\max}(A)\le 2^{|A_1|+|A_2|/2}=2^{|A_1|+((1/2-\ga)n-|A_1|)/2}=2^{n/4+(|A_1|-\ga n)/2}\le 2^{n/4+o(n)},$$
where the last inequality follows since $|A_1|\le (\ga+o(1))n$.
\end{proof}

\begin{lemma}\label{lem-odd}
If $A\in \cF$ such that $|A\setminus O|= o(n)$, then
$$f_{\max}(A)\le 2^{(1/4+o(1))n}.$$
\end{lemma}
\begin{proof}
Let $A\in\cF$ be as in the statement of the lemma. Notice that  if $S\subseteq T \subseteq [n]$ then $f_{\max}(S)\le f_{\max}(T)$. 
Using this fact, we may assume that $A=O\cup C$ with $C\subseteq E$ and $|C|= o(n)$. Similarly to before, every maximal sum-free subset of $[n]$ in $A$ can be built from choosing a sum-free set $S\subseteq C$ (at most $2^{|C|}= 2^{o(n)}$ choices) and extending $S$ in $O$ to a maximal one. Fix an arbitrary sum-free set $S$ in $C$ and let $G:=L_S[O]$ be the link graph of $S$ on vertex set $O$. Since $O$ is sum-free, by Lemma~\ref{claim-mis} we have that $N(S,O)\le {\rm{MIS}}(G)$.
It suffices to show that ${\rm{MIS}}(G)\leq 2^{n/4+o(n)}$. We will achieve this in two cases depending on the size of $S$.

\noindent\textbf{Case 1:} $|S|\geq n^{1/4}$.

In this case, we will show that $G$ is `not too sparse and almost regular'. Then we apply Lemma~\ref{lem-ik}.

We first show that $\de(G)\ge |S|/2$ and $\De(G)\le 2|S|+2$, thus $\De(G)\le 6\de(G)$. Let $x$ be any vertex in $O$. 
If $s \in S$ such that $s<\max\{x,n-x\}$ then at least one of $x-s$ and $x+s$ is adjacent to $x$ in $G$. If $s \in S$ 
such that $s\geq \max\{x,n-x\}$ then $s-x$ is adjacent to $x$ in $G$. By considering all $s \in S$ this implies that $\deg_G (x) \geq |S|/2$ (we divide by $2$ here as an edge $xy$ may arise from two different elements of $S$).
For the upper bound consider $x \in O$. If $xy \in E(G)$ then $y=x+s$, $x-s$ or $s-x$ for some $s \in S$ and only two of these terms are positive. Further, there may be a loop at $x$ in $G$ (contributing $2$ to the degree of $x$ in $G$). Thus, $\deg_G (x) \leq 2|S|+2$, as desired.

Since $\de(G)\ge |S|/2 \geq n^{1/4}/2$ we can apply Lemma~\ref{lem-ik} to $G$ with $k=6$. Hence,
$${\rm{MIS}}(G)\le 3^{\left(\frac{6}{7}\right)\frac{n/2}{3}+o(n)}\ll 2^{0.24n+o(n)}=o(2^{n/4}).$$


\noindent\textbf{Case 2:} $|S|\leq n^{1/4}$.

In this case, it suffices to show that $G$ has very few, $o(n)$, triangles, since then by applying Lemma~\ref{lem-ht} with $T$ being the vertex set of all triangles in $G$,  we have $|T|=o(n)$ and then ${\rm{MIS}}(G)\le 2^{n/4+o(n)}$. Recall that for each edge $xy$ in $G$, at least one of the evens $x+y$ and $|x-y|$ is in $S$. We call $xy$ a BLUE edge if $|x-y|$ is in $S$ and a RED edge if $|x-y|\not\in S$ and $x+y\in S$.
\begin{cl}\label{cl-type}
Each triangle in $G$ contains either 0 or 2 BLUE edges.
\end{cl}
\begin{proof}
Let $xyz$ be a triangle in $G$ with $x<y<z$. Suppose that $xyz$ has only one BLUE edge $xz$. Then $s_1:=z-x, s_2:=x+y$ and $s_3:=y+z$ are elements of $S$ and $s_1+s_2=s_3$, a contradiction to $S$ being sum-free. All other cases, including when all the edges are BLUE, are similar, we omit the proof here.
\end{proof}
Consider an arbitrary triple $\{s_1,s_2,s_3\}$ in $S$ (where $s_1,s_2$ and $s_3$ are not necessarily distinct). We say that $\{s_1,s_2,s_3\}$ \emph{forces a triangle $\mathcal T$ in $G$} if the vertex set $\{x,y,z\}$ of $\mathcal T$ is such that $s_1,x,y$; $s_2,y,z$ and; $s_3,x,z$ form Schur triples. Note that by definition of $G$, every triangle in $G$ is forced by some triple in $S$.

Fix an arbitrary triple $\{s_1,s_2,s_3\}$ in $S$. We will show that $\{s_1,s_2,s_3\}$ forces at most 24 triangles in $G$. This then implies that $G$ has at most $24|S|^3=o(n)$ triangles as desired.

By Claim~\ref{cl-type}, a triangle $xyz$ with $x<y<z$ can only be one of the following four types: (1) all edges are RED; (2) $xy$ is the only RED edge; (3) $yz$ is the only RED edge; (4) $xz$ is the only RED edge. 

It suffices to show that $\{s_1,s_2,s_3\}$ can force at most 6 triangles of each type. We show it only for Type (1), the other types are similar. 
Suppose that $xyz$ is a Type (1) triangle forced by $\{s_1,s_2,s_3\}$. Set   $M:=\left( \begin{array}{ccc}
1 & 1 & 0 \\
0 & 1 & 1 \\
1 & 0 & 1 \end{array} \right)$. Then $\mathbf{u}=(x,y,z)^T$ is a solution to 
 $M\cdot \mathbf{u}=\mathbf{s}$ for some $\mathbf{s}$ whose entries are precisely the elements of $\{s_1,s_2,s_3\}$.

Since $\det(M)=2\neq 0$, if a solution $\mathbf{u}$ exists to $M\cdot \mathbf{u}=\mathbf{s}$, it should be unique. The number of choices for $\mathbf{s}$, for fixed $\{s_1,s_2,s_3\}$, is $3!=6$.
Thus in total there are at most 6 triangles of Type (1) forced by $\{s_1,s_2,s_3\}$.

This completes the proof of Lemma~\ref{lem-odd}.
\end{proof}



\section*{Acknowledgements}
The authors are grateful to the referees for quick, careful and helpful reviews that led to significant improvements in the presentation of the paper.

This research was carried out whilst the second, third and fourth authors were visiting the Bolyai Institute, University of Szeged, Hungary. These authors thank the institute for the hospitality they received.
We are also grateful to Ping Hu for helpful discussions at the early stages of the project.

\end{document}